\documentclass[a4paper,12pt,egregdoesnotlikesansseriftitles,abstract,twoside=semi,headinclude=true,leqno]{scrartcl}
\usepackage{amsmath,amsthm,amssymb,ascmac,boites,xcolor,enumitem,scrlayer-scrpage}
\usepackage[T1]{fontenc}
\usepackage{lmodern,exscale}
\usepackage{bm}
\usepackage[colorlinks, linkcolor=purple, citecolor=teal]{hyperref}

\newtheorem{theorem}{Theorem}[section]
\newtheorem*{theorem*}{Theorem}

\newtheorem*{proposition*}{Proposition}
\newtheorem{lemma}[theorem]
{Lemma}
\newtheorem*{lemma*}{Lemma}

\newtheorem*{corollary*}{Corollary}

\newtheorem*{note*}{Note}

\newtheorem*{fact*}{Fact}
\theoremstyle{definition}
\newtheorem{definition}[theorem]
{Definition}
\newtheorem*{definition*}{Definition}
\newtheorem*{example*}{Example}

\newtheorem*{problem*}{Problem}

\newtheorem{axiom*}{Axiom}
\newtheorem*{notation*}{Notation}

\theoremstyle{remark}
\newtheorem{remark}[theorem]
{Remark}
\newtheorem*{remark*}{Remark}

\let\mb\mathbb
\let\mc\mathcal

\let\ve\varepsilon

\let\pt\partial

\newcommand{\ce}{\mathrel{\mathop:}=} 

\numberwithin{equation}{section}

\title
{\LARGE{Strong instability of standing waves \\
with negative energy for double power \\
nonlinear Schr\"odinger equations}}
\author
{Noriyoshi Fukaya and Masahito Ohta}
\date{}

\addtokomafont{pagehead}{\upshape}
\cehead{N. Fukaya and M. Ohta}
\cohead[]{Strong instability of standing waves}
\lehead*[]{\pagemark}
\rohead[]{\pagemark}
\lefoot[]{}
\rofoot[]{}

\begin{document}
\maketitle
\begin{abstract}
We study the strong instability of ground-state standing waves $e^{i\omega t}\phi_\omega(x)$ for $N$-dimensional nonlinear Schr\"odinger equations with double power nonlinearity.
One is $L^2$-subcritical,
and the other is $L^2$-supercritical.
The strong instability of standing waves with positive energy was proven by Ohta and Yamaguchi~(2015).
In this paper,
we improve the previous result, 
that is,
we prove that if $\partial_\lambda^2S_\omega(\phi_\omega^\lambda)|_{\lambda=1}\le0$,
the standing wave is strongly unstable,
where $S_\omega$ is the action,
and $\phi_\omega^\lambda(x)\ce\lambda^{N/2}\phi_\omega(\lambda x)$ is the $L^2$-invariant scaling.
\end{abstract}
\section{Introduction}
In this paper,
we consider the nonlinear Schr\"odinger equation with double power nonlinearity
\[ \label{nls}\tag{NLS}
i\pt_tu
=-\Delta u
-a|u|^{p-1}u
-b|u|^{q-1}u,\quad
(t,x)\in\mathbb{R}\times\mathbb{R}^N,
\]
where 
\begin{equation}\label{asmp}
N\in\mb{N},\quad
a>0,\quad
b>0,\quad
1<p<1+\frac{4}{N}<q<1+\frac{4}{N-2},
\end{equation}
and $u\colon\mb{R}\times\mb{R}^N\to\mb{C}$ is the unknown function of $(t,x)\in\mb{R}\times\mb{R}^N$.
Here, $1+4/(N-2)$ stands for $\infty$ if $N=1$ or $2$.
Eq.~\eqref{nls} appears in various regions of mathematical physics
(see \cite{BGMP89,Fibich,SS} and references therein).

The Cauchy problem for \eqref{nls} is locally well-posed in the energy space $H^1(\mb{R}^N)$ (see, e.g., \cite{Cazenave,Kato87}),
that is,
for each $u_0\in H^1(\mb{R}^N)$,
there exist the maximal lifespan $T_{\max}=T_{\max}(u_0)\in(0,\infty]$ and a unique solution $u\in C([0,T_{\max}),H^1(\mb{R}^N))$ of \eqref{nls} with $u(0)=u_0$ such that if $T_{\max}<\infty$,
then $\lim_{t\nearrow T_{\max}}\|\nabla u(t)\|_{L^2}=\infty$.
In the case $T_{\max}<\infty$,
we say that the solution $u(t)$ \textit{blows up in finite time}.
Moreover,
\eqref{nls} satisfies the two conservation laws
\[
E(u(t))
=E(u_0),\quad
\|u(t)\|_{L^2}
=\|u_0\|_{L^2}
\]
for all $t\in[0,T_{\max})$,
where $E$ is the energy defined by
\[
E(v)
=\frac12\|\nabla v\|_{L^2}^2
-\frac{a}{p+1}\|v\|_{L^{p+1}}^{p+1}
-\frac{b}{q+1}\|v\|_{L^{q+1}}^{q+1}.
\]
Furthermore, if
\begin{equation}\label{Sigma}
u_0
\in\Sigma
\ce\{\,v\in H^1(\mb{R}^N)\mid
\|xv\|_{L^2}<\infty\,\},
\end{equation}
then the solution $u(t)$ of \eqref{nls} with $u(0)=u_0$ belongs to $C([0,T_{\max}),\Sigma)$
and satisfies the virial identity
\begin{equation}\label{virial}
\frac{d^2}{dt^2}\|xu(t)\|_{L^2}^2
=8Q(u(t))
\end{equation}
for all $t\in[0,T_{\max})$
(see \cite[Section~6.5]{Cazenave}), where 
$v^\lambda(x)
=\lambda^{N/2}v(\lambda x)$ and
\begin{align}\label{defQ}
Q(v)
&=\pt_{\lambda}S_\omega(v^\lambda)|_{\lambda=1} \\ \notag
&=\|\nabla v\|_{L^2}^2
-\frac{aN(p-1)}{2(p+1)}\|v\|_{L^{p+1}}^{p+1}
-\frac{bN(q-1)}{2(q+1)}\|v\|_{L^{q+1}}^{q+1}.
\end{align}

Eq.\ \eqref{nls} has standing wave solutions of the form $e^{i\omega t}\phi(x)$,
where $\omega>0$ and $\phi\in H^1(\mb{R}^N)$ is a nontrivial solution of the stationary equation
\begin{equation}\label{se}
-\Delta\phi
+\omega\phi
-a|\phi|^{p-1}\phi
-b|\phi|^{q-1}\phi
=0,\quad
x\in\mb{R}^N.
\end{equation}
Eq.\ \eqref{se} can be rewritten as $S_\omega'(\phi)=0$,
where $S_\omega$ is the action defined by
\begin{align*}
S_\omega(v)
&=E(v)+\frac{\omega}{2}\|v\|_{L^2}^2 \\
&=\frac12\|\nabla v\|_{L^2}^2
+\frac\omega2\|v\|_{L^2}^2
-\frac a{p+1}\|v\|_{L^{p+1}}^{p+1}
-\frac b{q+1}\|v\|_{L^{q+1}}^{q+1}.
\end{align*}
It is known that if $\omega>0$,
then \eqref{se} has \textit{ground state} solutions, that is, the set 
\[
\mc{G}_\omega
\ce\{\,\phi\in\mc{F}_\omega\mid
S_\omega(\phi)\le S_\omega(v)~\text{for all}~v\in\mc{F}_\omega\,\}
\]
of nontrivial solutions to \eqref{se} with the minimal action is not empty
(see, e.g., \cite{BL83,Lions84,Strauss77}),
where
\[
\mc{F}_\omega
\ce\{\,v\in H^1(\mb{R}^N)\setminus\{0\}\mid
S_\omega'(v)=0\,\}
\]
is the set of all nontrivial solutions of \eqref{se}.

The stability and instability of standing waves are defined as follows.

\begin{definition}
Let $\phi\in\mc{F}_\omega$ be a nontrivial solution of \eqref{se}.
\begin{itemize}
\item
We say that the standing wave solution $e^{i\omega t}\phi$ of \eqref{nls} is \textit{stable} if for each $\ve>0$,
there exists $\delta>0$ such that if $u_0\in H^1(\mb{R}^N)$ satisfies $\|u_0-\phi\|_{H^1}<\delta$,
then the solution $u(t)$ of \eqref{nls} with $u(0)=u_0$ exists globally in time and satisfies
\[
\sup_{t\ge0}\inf_{(\theta,y)\in\mb{R}\times\mb{R}^N}\|u(t)-e^{i\theta}\phi(\cdot-y)\|_{H^1}<\ve.
\]
\item
We say that the standing wave solution $e^{i\omega t}\phi$ of \eqref{nls} is \textit{unstable} if it is not stable.
\item
We say that the standing wave solution $e^{i\omega t}\phi$ of \eqref{nls} is \textit{strongly unstable} if for each $\ve>0$,
there exists $u_0\in H^1(\mb{R}^N)$ such that $\|u_0-\phi\|_{H^1}<\delta$,
and the solution $u(t)$ of \eqref{nls} with $u(0)=u_0$ blows up in finite time.
\end{itemize}
\end{definition}

In this paper, we study the strong instability of the standing wave solution $e^{i\omega t}\phi_\omega$ for \eqref{nls},
where $\omega>0$, 
and $\phi_\omega\in\mc{G}_\omega$ is a ground state.

In the single power and 
$L^2$-critical or $L^2$-supercritical case when $a=0$,
$b>0$,
and $1+4/N\le q<1+4/(N-2)$,
Berestycki and Cazenave~\cite{BC81} proved that the standing wave is strongly unstable for any $\omega>0$
(see also \cite{Weinstein8283} for the case $q=1+4/N$),
whereas
in $L^2$-subcritical case when $a>0$,
$b=0$, and
$1<p<1+4/N$,
Cazenave and Lions~\cite{CL82} proved that the standing wave is stable for any $\omega>0$.

In the double power case when \eqref{asmp} is assumed,
the argument of Ohta~\cite{Ohta95-1} showed the instability of standing waves for sufficiently large $\omega>0$.
In \cite{Ohta95-1},
he proved that if $\pt_\lambda^2S_\omega(\phi_\omega^\lambda)|_{\lambda=1}<0$,
then the standing wave is unstable,
where $v^\lambda(x)\ce\lambda^{N/2}v(\lambda x)$ is the scaling,
which does not change the $L^2$-norm.
On the other hand,
Fukuizumi~\cite{Fukuizumi03} proved
the stability of standing waves for sufficiently small $\omega>0$.
See also \cite{Maeda08,Ohta95-d} for the stability and instability in one dimensional case.
The strong instability of standing waves for sufficiently large $\omega$ was proven by Ohta and Yamaguchi \cite{OY15}.
In \cite{OY15},
they proved the strong instability of standing waves with positive energy $E(\phi_\omega)>0$ by using and modifying the idea of Zhang~\cite{Zhang02} and Le Coz~\cite{LeCoz08-1}
(see also \cite{OY16} for related works).

Recently,
for the nonlinear Schr\"odinger equation with harmonic potential,
Ohta~\cite{Ohta18} proved that if $\pt_\lambda^2\tilde S_\omega(\phi_\omega^\lambda)|_{\lambda=1}\le0$, then the standing waves is strongly unstable,
where $\tilde S_\omega$ is the corresponding action.
This assumption is the same one as in Ohta~\cite{Ohta95-1}.
More recently,
Fukaya and Ohta~\cite{FOpre} proved the strong instability of standing waves for nonlinear Schr\"odinger equation with an attractive inverse power potential
\begin{equation}\label{nlsi}
i\pt_tu
=-\Delta u
-\frac{\gamma}{|x|^\alpha}u
-|u|^{q-1}u,\quad
(t,x)\in\mb{R}\times\mb{R}^N
\end{equation}
with $\gamma>0$, $0<\alpha<\min\{2,N\}$, and $1+4/N<q<1+4/(N-2)$
under the same assumption $\pt_\lambda^2\tilde S_\omega(\phi_\omega^\lambda)|_{\lambda=1}\le0$ as in \cite{Ohta18} by using the idea of Ohta \cite{Ohta18} with some modifications.


For \eqref{nls},
the strong instability of standing waves with negative energy was not known.
The aim of this paper is to prove the strong instability under the same assumption $\pt_\lambda^2S_\omega(\phi_\omega^\lambda)|_{\lambda=1}\le0$ as in \cite{FOpre,Ohta18}.
Now, we state our main result.

\begin{theorem}\label{mainthm}
Assume \eqref{asmp}, $\omega>0$,
and that $\phi_\omega\in\mc{G}_\omega$ satisfies $\pt_\lambda^2S_\omega(\phi_\omega^\lambda)|_{\lambda=1}\le0$,
where $\phi_\omega^\lambda(x)=\lambda^{N/2}\phi_\omega(\lambda x)$.
Then the standing wave solution $e^{i\omega t}\phi_\omega$ of \eqref{nls} is strongly unstable.
\end{theorem}

\begin{remark}
In the case \eqref{asmp},
$E(\phi_\omega)>0$ implies $\pt_\lambda^2S_\omega(\phi_\omega^\lambda)|_{\lambda=1}<0$.
Indeed, 
let
$\alpha=N(p-1)/2$ and
$\beta=N(q-1)/2$.
Then since $Q(\phi_\omega)=\pt_{\lambda}S_\omega(\phi_\omega^\lambda)|_{\lambda=1}=0$ and
$0<\alpha<2<\beta$,
we have
\begin{align*}
\pt_\lambda^2S_\omega(\phi_\omega^\lambda)|_{\lambda=1}
&=\|\nabla\phi_\omega\|_{L^2}^2
-\frac{a\alpha(\alpha-1)}{p+1}\|\phi_\omega\|_{L^{p+1}}^{p+1}
-\frac{b\beta(\beta-1)}{q+1}\|\phi_\omega\|_{L^{q+1}}^{q+1} \\
&=(\alpha+1)Q(\phi_\omega)
-2\alpha E(\phi_\omega)
-\frac{b(\beta-2)(\beta-\alpha)}{q+1}\|\phi_\omega\|_{L^{q+1}}^{q+1} \\
&<0.
\end{align*}
Therefore,
Theorem~\ref{mainthm} is an improvement of the result of Ohta and Yamaguchi~\cite{OY15}.
\end{remark}

To prove Theorem~\ref{mainthm},
we introduce the set
\[
\mc{B}_\omega
\ce\left\{v\in H^1(\mb{R}^N)\mathrel{}\middle|\mathrel{}
\begin{alignedat}{2}
&S_\omega(v)<S_\omega(\phi_\omega),~& &
\|v\|_{L^2}\le\|\phi_\omega\|_{L^2}, \\
&K_\omega(v)<0,~& &
Q(v)<0
\end{alignedat}
\right\},
\]
where
\begin{equation}\label{Fv}
\begin{aligned}
K_\omega(v)
&\ce\pt_\lambda S_\omega(\lambda v)|_{\lambda=1} \\
&=\|\nabla v\|_{L^2}^2
+\omega\|v\|_{L^2}^2
-a\|v\|_{L^{p+1}}^{p+1}
-b\|v\|_{L^{q+1}}^{q+1}
\end{aligned}
\end{equation}
is the Nehari functional.
Then we obtain the following blowup result.

\begin{theorem} \label{blowup}
Assume \eqref{asmp}, $\omega>0$,
and that $\phi_\omega\in\mc{G}_\omega$ satisfies $\pt_\lambda^2S_\omega(\phi_\omega^\lambda)|_{\lambda=1}\le0$.
Let $u_0\in\mc{B}_\omega\cap\Sigma$.
Then the solution $u(t)$ of \eqref{nls} with $u(0)=u_0$ blows up in finite time.
\end{theorem}

Theorem~\ref{mainthm} follows from Theorem~\ref{blowup} because the scaling of the ground state $\phi_\omega^\lambda$ belongs to $\mc{B}_\omega\cap\Sigma$ for all $\lambda>1$
(see Section~\ref{sec:si} below).

The proof of Theorem~\ref{blowup} is based on the variational argument in Ohta~\cite{Ohta18} and Fukaya and Ohta \cite{FOpre}.
Firstly, we derive the key estimate
$Q(v)/2\le S_\omega(v)-S_\omega(\phi_\omega)$
for all $v\in\mc{B}_\omega$
(Lemma~\ref{keylem}).
Then by using the conservation laws,
the variational characterization of the ground state by the Nehari functional,
and the key estimate,
we show the invariance of $\mc{B}_\omega$ under the flow of \eqref{nls} (Lemma~\ref{invariant}).
Combining the virial identity with the key estimate, finally, we can obtain blowup of solutions to \eqref{nls} with initial data belonging to $\mc{B}_\omega\cap\Sigma$ by the classical argument as in Berestycki and Cazenave \cite{BC81}.

We prove the key estimate $Q/2\le S_\omega-S_\omega(\phi_\omega)$ on $\mc{B}_\omega$ following the proof of the same estimate for \eqref{nlsi} in \cite[Lemma~3.2]{FOpre}.
The proof relies on the variational characterization of the ground state by the Nehari functional
\[
S_\omega(\phi_\omega)
=\inf\{\, S_\omega(v)\mid
v\in H^1(\mb{R}^N)\setminus\{0\},~
K_\omega(v)=0\,\}
\]
and the property of the graph of the function $\lambda\mapsto S_\omega(v^\lambda)$.
Note that the graph of $S_\omega(v^\lambda)$ for \eqref{nls} has the same property as that for \eqref{nlsi}.
In the case of \eqref{nlsi},
since the action $\tilde S_\omega$ can be expressed by using the Nehari functional $\tilde{K}_\omega(v)\ce\pt_{\lambda}\tilde S_\omega(\lambda v)|_{\lambda=1}$ as
\begin{equation}\label{skq}
\tilde S_\omega(v)
=\frac12\tilde K_\omega(v)+\frac{(q-1)}{2(q+1)}\|v\|_{L^{q+1}}^{q+1},
\end{equation}
the above variational characterization can be written by using $L^{q+1}$-norm.
Therefore, in \cite{FOpre},
they used not only the action but also $L^{q+1}$-norm effectively.

On the other hand,
in the case of \eqref{nls},
the action $S_\omega$ cannot be expressed as \eqref{skq} because \eqref{nls} has double power nonlinearity.
Due to this fact, we can not directly apply the proof in \cite{FOpre}.
However, in this case,
we see that the action can be expressed as
\[
S_\omega(v)
=\frac12K_\omega(v)
+\frac12F(v),
\]
where
\[
F(v)=\frac{a(p-1)}{2(p+1)}\|v\|_{L^{p+1}}^{p+1}
+\frac{b(q-1)}{2(q+1)}\|v\|_{L^{q+1}}^{q+1}.
\]
Therefore, we can use $F$ instead of $L^{q+1}$-norm.
By applying the argument in \cite{FOpre} using $F$,
although the calculation processes differ from that in \cite{FOpre},
we can prove the key estimate above.

We finally remark that in fact, the assumption 
$\pt_\lambda^2S_\omega(\phi_\omega^\lambda)|_{\lambda=1}\le0$ is not a sufficient condition for the instability of standing waves
(see \cite[Section~4]{OY16} for related remarks).
However, in \cite{FOpre,Ohta18} and this paper, 
this assumption plays a very important role in the proof of the strong instability of standing waves.
It seems an interesting problem whether the unstable standing wave is strongly unstable or not if the assumption 
$\pt_\lambda^2S_\omega(\phi_\omega^\lambda)|_{\lambda=1}\le0$
is broken.

The rest of this paper is organized as follows.
In Section~\ref{sec:blowup},
we prove Theorem~\ref{blowup}, 
that is,
we prove that if $\pt_\lambda^2S_\omega(\phi_\omega^\lambda)|_{\lambda=1}\le0$,
then the solution of \eqref{nls} with $u(0)=u_0\in\mc B_\omega\cap\Sigma$ blows up in finite time.
In Section~\ref{sec:si},
we prove the strong instability of standing waves by using Theorem~\ref{blowup}.


\section{Blowup}\label{sec:blowup}

In this section,
we prove Theorem~\ref{blowup}.
Throughout this section,
we assume \eqref{asmp} and $\omega>0$.
Recall that the ground state $\phi_\omega\in\mc{G}_\omega$
satisfies $K_\omega(\phi_\omega)=0$ and the variational characterization
\begin{equation}\label{vari-char4}
S_\omega(\phi_\omega)
=\inf\{\, S_\omega(v)\mid
v\in H^1(\mb{R}^N)\setminus\{0\},~
K_\omega(v)=0\,\}
\end{equation}
(see, e.g., \cite{LeCoz08-2,Lions84}),
where $K_\omega$ is the Nehari functional defined in \eqref{Fv}.

Firstly,
we prove the key lemma in the proof.
Note that the action $S_\omega$ is expressed as
\begin{equation}\label{skf}
S_\omega(v)
=\frac12K_\omega(v)
+\frac12F(v),
\end{equation}
where 
\[
F(v)
=\frac{a(p-1)}{p+1}\|v\|_{L^{p+1}}^{p+1}
+\frac{b(q-1)}{q+1}\|v\|_{L^{q+1}}^{q+1}.
\]
Therefore, the characterization \eqref{vari-char4} is rewritten as
\begin{equation}\label{vari-char5}
S_\omega(\phi_\omega)
=\frac12F(\phi_\omega)
=\inf\left\{\,\frac12F(v)\mathrel{}\middle|\mathrel{}
v\ne0,~
K_\omega(v)=0\,\right\}.
\end{equation}

Let
\[
\alpha=\frac{N(p-1)}{2},\quad
\beta=\frac{N(q-1)}{2}.
\]
Using this notation,
we have
\begin{align*}
S_\omega(v^\lambda)
&=\frac{\lambda^2}2\|\nabla v\|_{L^2}^2
+\frac{\omega}{2}\|v\|_{L^2}^2
-\frac{a\lambda^\alpha}{p+1}\|v\|_{L^{p+1}}^{p+1}
-\frac{b\lambda^\beta}{q+1}\|v\|_{L^{q+1}}^{q+1}, \\
K_\omega(v^\lambda)
&=\lambda^2\|\nabla v\|_{L^2}^2
+\omega\|v\|_{L^2}^2
-a\lambda^\alpha\|v\|_{L^{p+1}}^{p+1}
-b\lambda^\beta\|v\|_{L^{q+1}}^{q+1}, \\
\frac{N}{2}F(v^{\lambda})
&=\frac{a\alpha\lambda^\alpha}{p+1}\|v\|_{L^{p+1}}^{p+1}
+\frac{b\beta\lambda^\beta}{q+1}\|v\|_{L^{q+1}}^{q+1}, \\
Q(v)
&=\|\nabla v\|_{L^2}^2
-\frac{a\alpha}{p+1}\|v\|_{L^{p+1}}^{p+1}
-\frac{b\beta}{q+1}\|v\|_{L^{q+1}}^{q+1}, \\
\pt_\lambda^2S_\omega(v^\lambda)|_{\lambda=1}
&=\|\nabla v\|_{L^2}^2
-\frac{a\alpha(\alpha-1)}{p+1}\|v\|_{L^{p+1}}^{p+1}
-\frac{b\beta(\beta-1)}{q+1}\|v\|_{L^{q+1}}^{q+1},
\end{align*}
where $v^\lambda(x)=\lambda^{N/2}v(\lambda x)$.
Note that by $S_\omega'(\phi_\omega)=0$,
\[
K_\omega(\phi_\omega)
=\langle S_\omega'(\phi_\omega),\phi_\omega\rangle
=0,\quad
Q(\phi_\omega)
=\langle S_\omega'(\phi_\omega),\pt_\lambda\phi_\omega^\lambda|_{\lambda=1}\rangle
=0.
\]

\begin{lemma}\label{keylem}
Assume that $\phi_\omega\in\mc{G}_\omega$ satisfies $\pt_\lambda^2S_\omega(\phi_\omega^\lambda)|_{\lambda=1}\le0$.
Let $v\in H^1(\mb{R}^N)$ satisfy
\[
v\ne0,\quad
\|v\|_{L^2}^2\le\|\phi_\omega\|_{L^2}^2,\quad
K_\omega(v)\le0,\quad
Q(v)\le0.
\]
Then
\begin{equation*}
\frac{Q(v)}{2}
\le S_\omega(v)
-S_\omega(\phi_\omega).
\end{equation*}
\end{lemma}

\begin{proof}
Since $\lim_{\lambda\searrow0}K_\omega(v^\lambda)=\omega\|v\|_{L^2}^2>0$  and $K_\omega(v)\le0$,
there exists $\lambda_0\in(0,1]$ such that
$K_\omega(v^{\lambda_0})=0$.
By the definition of the scaling $v^\lambda$ and \eqref{vari-char5},
we have
\begin{align}\label{L^2l0}
&\|v^{\lambda_0}\|_{L^2}
=\|v\|_{L^2}
\le\|\phi_\omega\|_{L^2}, \\ \label{Fl0}
&\frac{N}{2}F(\phi_\omega)
\le\frac{N}{2}F(v^{\lambda_0})
=\frac{a\alpha\lambda_0^\alpha}{p+1}\|v\|_{L^{p+1}}^{p+1}
+\frac{b\beta\lambda_0^\beta}{q+1}\|v\|_{L^{q+1}}^{q+1}.
\end{align}
Now, we define
\begin{align*}
f(\lambda)
&=S_\omega(v^\lambda)
-\frac{\lambda^2}{2}Q(v) \\
&=\frac{\omega}{2}\|v\|_{L^2}^2
-\frac{a}{p+1}\left(\lambda^\alpha-\frac{\alpha\lambda^2}{2}\right)\|v\|_{L^{p+1}}^{p+1}
-\frac{b}{q+1}\left(\lambda^\beta-\frac{\beta\lambda^2}{2}\right)\|v\|_{L^{q+1}}^{q+1}.
\end{align*}
If we have $f(\lambda_0)\le f(1)$,
then by \eqref{vari-char4} and $Q(v)\le0$,
we obtain
\begin{equation}\label{aim2}
S_\omega(\phi_\omega)
\le S_\omega(v^{\lambda_0})
\le S_\omega(v^{\lambda_0})-\frac{\lambda_0^2}{2}Q(v)
\le S_\omega(v)-\frac{Q(v)}{2}.
\end{equation}
This is the desired inequality.

In what follows, we prove the inequality $f(\lambda_0)\le f(1)$.
This is equivalent to
\begin{equation}\label{aim}
\frac{a}{p+1}\|v\|_{L^{p+1}}^{p+1}
\le\frac{b}{q+1}\cdot\frac{2\lambda_0^\beta-\beta\lambda_0^2-2+\beta}{\alpha\lambda_0^2-2\lambda_0^\alpha-\alpha+2}\|v\|_{L^{q+1}}^{q+1}.
\end{equation}
Since 
\begin{equation}\label{alphabeta}
\frac{p+1}\alpha+\frac2\beta
=\frac2N+\frac2\beta+\frac2\alpha
=\frac{q+1}\beta+\frac2\alpha,
\end{equation}
we have
\begin{align*}
&K_\omega(\phi_\omega)
+\frac{2}{\alpha\beta}\pt_{\lambda}^2S_\omega(\phi_\omega^\lambda)|_{\lambda=1}
-\left(1+\frac{2}{\alpha\beta}\right)Q(\phi_\omega) \\
&\quad=\omega\|\phi_\omega\|_{L^2}^2
-\frac{a\alpha}{p+1}\left(\frac{p+1}{\alpha}+\frac2\beta-1-\frac4{\alpha\beta}\right)\|\phi_\omega\|_{L^{p+1}}^{p+1} \\
&\qquad-\frac{b\beta}{q+1}\left(\frac{q+1}{\beta}+\frac2\alpha-1-\frac4{\alpha\beta}\right)\|\phi_\omega\|_{L^{q+1}}^{q+1} \\
&\quad=\omega\|\phi_\omega\|_{L^2}^2
-\left(\frac{q+1}{\beta}+\frac2\alpha-1-\frac4{\alpha\beta}\right)\frac{N}{2}F(\phi_\omega).
\end{align*}
Therefore, by $K_\omega(\phi_\omega)=Q(\phi_\omega)=0$ and the assumption $\pt_\lambda^2S_\omega(\phi_\omega^\lambda)|_{\lambda=1}\le0$,
we obtain
\[
\omega\|\phi_\omega\|_{L^2}^2
\le\left(\frac{q+1}{\beta}+\frac2\alpha-1-\frac4{\alpha\beta}\right)\frac{N}{2}F(\phi_\omega).
\]
Combining \eqref{L^2l0} and \eqref{Fl0} with this inequality,
and using \eqref{alphabeta} again,
it follows that
\begin{equation}\label{2pq}
\begin{aligned}
\omega\|v\|_{L^2}^2
&\le\left(a
+\frac{a}{p+1}\cdot\frac1\beta\left(2\alpha-\alpha\beta-4\right)\right)\lambda_0^\alpha\|v\|_{L^{p+1}}^{p+1} \\
&\quad+\left(b+\frac{b}{q+1}\cdot\frac1\alpha\left(2\beta-\alpha\beta-4\right)\right)\lambda_0^\beta\|v\|_{L^{q+1}}^{q+1}.
\end{aligned}
\end{equation}
Moreover, by $K_\omega(v^{\lambda_0})=0$,
$Q(v)\le0$, and
\eqref{2pq},
we deduce
\begin{align*}
a\|v\|_{L^{p+1}}^{p+1}
&=\lambda_0^{2-\alpha}\|\nabla v\|_{L^2}^2
+\lambda_0^{-\alpha}\omega\|v\|_{L^2}^2
-b\lambda_0^{\beta-\alpha}\|v\|_{L^{q+1}}^{q+1} \\
&\le\lambda_0^{2-\alpha}\left(\frac{a\alpha}{p+1}\|v\|_{L^{p+1}}^{p+1}+\frac{b\beta}{q+1}\|v\|_{L^{q+1}}^{q+1}\right) \\
&\quad+\left(a+\frac{a}{p+1}\cdot\frac1{\beta}\left(2\alpha-\alpha\beta-4\right)\right)\|v\|_{L^{p+1}}^{p+1} \\
&\quad+\left(b+\frac{b}{q+1}\cdot\frac1{\alpha}\left(2\beta-\alpha\beta-4\right)\right)\lambda_0^{\beta-\alpha}\|v\|_{L^{q+1}}^{q+1} 
-b\lambda_0^{\beta-\alpha}\|v\|_{L^{q+1}}^{q+1} \\
&=\left(a+\frac{a}{p+1}\cdot\frac{1}{\beta}\left(2\alpha-\alpha\beta-4+\alpha\beta\lambda_0^{2-\alpha}\right)\right)\|v\|_{L^{p+1}}^{p+1} \\
&\quad+\frac{b}{q+1}\cdot\frac{1}{\alpha}\left(\left(2\beta-\alpha\beta-4\right)\lambda_0^{\beta-\alpha}+\alpha\beta\lambda_0^{2-\alpha}\right)\|v\|_{L^{q+1}}^{q+1},
\end{align*} 
and thus
\begin{align*}
&\frac{a}{p+1}\cdot\frac{1}{\beta}\left(\alpha\beta+4-2\alpha-\alpha\beta\lambda_0^{2-\alpha}\right)\|v\|_{L^{p+1}}^{p+1} \\
&\quad\le\frac{b}{q+1}\cdot\frac{1}{\alpha}\left(\left(2\beta-\alpha\beta-4\right)\lambda_0^{\beta-\alpha}+\alpha\beta\lambda_0^{2-\alpha}\right)\|v\|_{L^{q+1}}^{q+1}.
\end{align*}
Since 
$\alpha\beta+4-2\alpha-\alpha\beta\lambda_0^{2-\alpha}
\ge4-2\alpha>0$,
this is rewritten as
\begin{equation}\label{Lpq}
\frac{a}{p+1}\|v\|_{L^{p+1}}^{p+1}
\le\frac{b}{q+1}
\cdot\frac{\beta(2\beta-\alpha\beta-4)\lambda_0^{\beta-\alpha}+\alpha\beta^2\lambda_0^{2-\alpha}}{\alpha(\alpha\beta+4-2\alpha-\alpha\beta\lambda_0^{2-\alpha})}
\|v\|_{L^{q+1}}^{q+1}.
\end{equation}
In view of \eqref{aim} and \eqref{Lpq},
it suffices to show that
\[
\frac{\beta(2\beta-\alpha\beta-4)\lambda_0^{\beta-\alpha}+\alpha\beta^2\lambda_0^{2-\alpha}}{\alpha(\alpha\beta+4-2\alpha-\alpha\beta\lambda_0^{2-\alpha})}
\le\frac{2\lambda_0^\beta-\beta\lambda_0^2-2+\beta}{\alpha\lambda_0^2-2\lambda_0^\alpha-\alpha+2}.
\]
This inequality follows if we have
\begin{align*}
g_1(\lambda)
&\ce\frac{\alpha(2\lambda^\beta-\beta\lambda^2-2+\beta)(\alpha\beta+4-2\alpha-\alpha\beta\lambda^{2-\alpha})}{(\alpha\lambda^2-2\lambda^\alpha-\alpha+2)\lambda^{\beta-\alpha}}\\
&\quad-\beta(2\beta-\alpha\beta-4)
-\frac{\alpha\beta^2}{\lambda^{\beta-2}} \\
&\ge0
\end{align*}
for all $\lambda\in(0,1)$.
Since $\lim_{\lambda\nearrow1}g_1(\lambda)=0$,
it is enough to show that $g_1'(\lambda)\le0$ for all $\lambda\in(0,1)$.
A direct calculation shows 
\begin{align*}
g_1'(\lambda)
&=\frac{a\lambda^{\alpha-\beta+1}}{(\alpha\lambda^2-2\lambda^\alpha-\alpha+2)^2} \\
&\cdot\big((2-\alpha)(\beta-2)
-2\beta\lambda^{-\alpha}
+(\alpha\beta-2\alpha+4)\lambda^{-2}\big) \\
&\cdot\big(2\alpha(2-\alpha)\lambda^\beta
-\alpha\beta(\beta-\alpha)\lambda^2
+2\beta(\beta-2)\lambda^\alpha
-(2-\alpha)(\beta-2)(\beta-\alpha)\big).
\end{align*}
Now, we put
\begin{align*}
h(\lambda)
=(2-\alpha)(\beta-2)
-2\beta\lambda^{-\alpha}
+(\alpha\beta-2\alpha+4)\lambda^{-2}.
\end{align*}
Since $h(1)=0$ and for $\lambda\in(0,1)$,
\[
h'(\lambda)
=-2\alpha\beta(\lambda^{-3}-\lambda^{-\alpha-1})
-4(2-\alpha)\lambda^{-3}\le0,
\]
we have $h(\lambda)\ge0$.
Thus, we only have to show that
\[
g_2(\lambda)
\ce2\alpha(2-\alpha)\lambda^\beta
-\alpha\beta(\beta-\alpha)\lambda^2
+2\beta(\beta-2)\lambda^\alpha
-(2-\alpha)(\beta-2)(\beta-\alpha)\le0
\]
for all $\lambda\in(0,1)$.
Since $g_2(1)=0$,
it suffices to show that
\[
g_2'(\lambda)
=2\alpha\beta\lambda^{\alpha-1}\left(
(2-\alpha)\lambda^{\beta-\alpha}
-(\beta-\alpha)\lambda^{2-\alpha}
+\beta-2
\right)
\ge0
\]
for all $\lambda\in(0,1)$.
This is equivalent to
\[
g_3(\lambda)
\ce(2-\alpha)\lambda^{\beta-\alpha}
-(\beta-\alpha)\lambda^{2-\alpha}
+\beta-2\ge0.
\]
Since $g_3(1)=0$,
and
\[
g_3'(\lambda)
=-(\beta-\alpha)(2-\alpha)\lambda^{1-\alpha}(1-\lambda^{\beta-2})
\le0
\]
for all $\lambda\in(0,1)$,
we obtain $g_3(\lambda)\ge0$ for all $\lambda\in(0,1)$.
This implies $f(\lambda_0)\le f(1)$.
Thus,
the inequality~\eqref{aim2} follows.
This completes the proof.
\end{proof}

Next, we show that the set $\mc{B}_\omega$ is invariant under the flow of \eqref{nls}.
Recall that the definition of $\mc{B}_\omega$ is given by
\[
\mc{B}_\omega
=\left\{v\in H^1(\mb{R}^N)\mathrel{}\middle|\mathrel{}
\begin{alignedat}{2}
&S_\omega(v)<S_\omega(\phi_\omega),~& &
\|v\|_{L^2}\le\|\phi_\omega\|_{L^2}, \\
&K_\omega(v)<0,~& &
Q(v)<0
\end{alignedat}
\right\}.
\]

\begin{lemma}\label{invariant}
Let $u_0\in\mc{B}_\omega$.
Then the solution $u(t)$ of \eqref{nls} with $u(0)=u_0$ belongs to $\mc{B}_\omega$ for all $t\in[0,T_{\max})$.
\end{lemma}

\begin{proof}
Since $S_\omega$ and $\|\cdot\|_{L^2}$ are the conserved quantities of \eqref{nls},
we have
$S_\omega(u(t))=S_\omega(u_0)<S_\omega(\phi_\omega)$
and $\|u(t)\|_{L^2}=\|u_0\|_{L^2}\le\|\phi_\omega\|_{L^2}$ for all $t\in[0,T_{\max})$.
Therefore, by \eqref{vari-char4},
we have $K_\omega(u(t))\ne0$ for all $t\in[0,T_{\max})$.
Moreover, by $K_\omega(u_0)<0$ and the continuity of the solution $u(t)$,
we obtain $K_\omega(u(t))<0$ for all $t\in[0,T_{\max})$.
Finally,
we show that $Q(u(t))<0$ for all $t\in[0,T_{\max})$.
If not,
there exists $t_0\in(0,T_{\max})$ such that $Q(u(t_0))=0$.
Then by Lemma~\ref{keylem} and $S_\omega(u(t_0))<S_\omega(\phi_\omega)$,
we have $Q(u(t_0))<0$.
This is a contradiction.
This completes the proof.
\end{proof}

Finally,
we prove the blowup result.

\begin{proof}[Proof of Theorem~\ref{blowup}]
By the virial identity~\eqref{virial},
Lemmas~\ref{keylem} and \ref{invariant},
and the conservation of $S_\omega$,
we have
\begin{align*}
\frac{d^2}{dt^2}\|xu(t)\|_{L^2}^2
&=8Q(u(t)) \\
&\le16\big(S_\omega(u(t))-S_\omega(\phi_\omega)\big)
=16\big(S_\omega(u_0)-S_\omega(\phi_\omega)\big)
<0
\end{align*}
for all $t\in[0,T_{\max})$.
This implies $T_{\max}<\infty$.
This completes the proof.
\end{proof}

\section{Strong instability}\label{sec:si}

In this section,
we prove Theorem~\ref{mainthm} using Theorem~\ref{blowup}.
Throughout this section,
we impose the assumption of Theorem~\ref{mainthm}.

We remark that
\begin{align*}
S_\omega(v^\lambda)
&=\frac12K_\omega(v^\lambda)
+\frac12F(v^\lambda) \\
&=\frac{\lambda^2}2\|\nabla v\|_{L^2}^2
+\frac{\omega}{2}\|v\|_{L^2}^2
-\frac{a\lambda^\alpha}{p+1}\|v\|_{L^{p+1}}^{p+1}
-\frac{b\lambda^\beta}{q+1}\|v\|_{L^{q+1}}^{q+1}, \\
Q(v^\lambda)
&=\lambda\pt_\lambda S_\omega(v^\lambda), \\
Q(\phi_\omega)
&=\pt_{\lambda}S_\omega(\phi_\omega^\lambda)|_{\lambda=1}
=0,\quad
\pt_{\lambda}^2S_\omega(\phi_\omega^\lambda)|_{\lambda=1}
\le0.
\end{align*}

\begin{lemma}\label{pol}
Assume that $\phi_\omega\in\mc{G}_\omega$ satisfies $\pt_\lambda^2S_\omega(\phi_\omega^\lambda)|_{\lambda=1}\le0$.
Then $\phi_\omega^\lambda\in\mc{B}_\omega$ for all $\lambda>1$.
\end{lemma}

\begin{proof}
By the definition of the scaling $\lambda\mapsto v^\lambda$,
we have 
$\|\phi_\omega^\lambda\|_{L^2}
=\|\phi_\omega\|_{L^2}$
for all $\lambda>1$.
Since $\pt_{\lambda}S_\omega(\phi_\omega^\lambda)|_{\lambda=1}
=0$ and 
$\pt_{\lambda}^2S_\omega(\phi_\omega^\lambda)|_{\lambda=1}
\le0$,
in view of the graph of $\lambda\mapsto S_\omega(\phi_\omega^\lambda)$,
we see that $S_\omega(\phi_\omega^\lambda)<S_\omega(\phi_\omega)$
and $Q(\phi_\omega^\lambda)=\lambda\pt_{\lambda}S_\omega(\phi_\omega^\lambda)<0$ for all $\lambda>1$.
Finally, we obtain
\[K_\omega(\phi_\omega^\lambda)
=2S_\omega(\phi_\omega^\lambda)-F(\phi_\omega^\lambda)
<2S_\omega(\phi_\omega)-F(\phi_\omega)
=0\] for all $\lambda>1$.
This completes the proof.
\end{proof}

Now,
we prove our main theorem.

\begin{proof}[Proof of Theorem~\ref{mainthm}]
By an analogous argument in the proof of \cite[Theorem~8.1.1]{Cazenave},
we see that $\phi_\omega$ decays exponentially.
This implies $\phi_\omega\in\Sigma$,
where $\Sigma$ is the weighted space defined in \eqref{Sigma}. 
Therefore, combining this with Lemma~\ref{pol},
we have $\phi_\omega^\lambda\in\mc{B}_\omega\cap\Sigma$ for all $\lambda>1$.
Thus, Theorem~\ref{blowup} implies that for any $\lambda>1$,
the solution $u(t)$ of \eqref{nls} with $u(0)=\phi_\omega^\lambda$ blows up in finite time.
Moreover, we obtain $\phi_\omega^\lambda\to\phi_\omega$ in $H^1(\mb{R}^N)$ as $\lambda\searrow1$.
Hence, the standing wave solution $e^{i\omega t}\phi_\omega$ of \eqref{nls} is strongly unstable.
\end{proof}

\section*{Acknowledgements}

The first author was supported by Grant-in-Aid for JSPS Fellows 18J11090.
The second author was supported by JSPS KAKENHI Grant Numbers 18K03379 and 26247013.

Noriyoshi Fukaya

Department of Mathematics, Graduate School of Science, Tokyo University of Science,
1-3 Kagurazaka, Shinjuku-ku, Tokyo 162-8601, Japan

\textit{E-mail address}: \texttt{1116702@ed.tus.ac.jp}

\vspace\baselineskip

Masahito Ohta

Department of Mathematics, Tokyo University of Science, 
1-3 Kagurazaka, Shinjuku-ku, Tokyo 162-8601, Japan 

\textit{E-mail address}: \texttt{mohta@rs.tus.ac.jp}


\begin{thebibliography}{99}

\bibitem{BGMP89}
I. V. Barashenkov, A. D. Gocheva, V. G. Makhankov, and I. V. Puzynin, 
\textit{Stability of the soliton-like ``bubbles''},  Phys.\ D \textbf{34} (1989), 240--254.

\bibitem{BC81}
H. Berestycki and T. Cazenave,
\textit{Instabilit\'e des \'etats stationaires dans les \'equations de Schr\"odinger et de Klein--Gordon non lin\'eaires}, 
C. R. Acad.\ Sci.\ Paris S\'er.\ I Math.\ \textbf{293} (1981),  489--492.

\bibitem{BL83}
H. Berestycki and P.-L. Lions,
\textit{Nonlinear scalar field equations, I, Existence of a ground state}, 
Arch.\ Rational Mech.\ Anal.\ \textbf{82} (1983), 313--345.


\bibitem{Cazenave}
T. Cazenave,
Semilinear Schr\"odinger equations, 
Courant Lecture Notes in Mathematics, 10. New York University, Courant Institute of Mathematical Sciences, New York; American Mathematical Society, Providence, RI, 2003.

\bibitem{CL82}
T. Cazenave and P.-L. Lions,
\textit{Orbital stability of standing waves for some nonlinear Schr\"odinger equations}, 
Comm.\ Math.\ Phys.\ \textbf{85} (1982), 549--561.

\bibitem{Fibich}
G. Fibich,
The nonlinear Schr\"odinger equation: Singular solutions and optical collapse, 
Applied Mathematical Sciences, \textbf{192}, Springer, Cham, 2015.

\bibitem{FOpre}
N. Fukaya and M. Ohta,
\textit{Strong instability of standing waves for nonlinear Schr\"odinger equations with inverse power potential},
preprint, arXiv:1804.02127.

\bibitem{Fukuizumi03}
R. Fukuizumi, 
\textit{Remarks on the stable standing waves for nonlinear Schr\"odinger equations with double power nonlinearity}, 
Adv.\ Math.\ Sci.\ Appl.\ \textbf{13} (2003),549--564.

%
%



%

\bibitem{Kato87}
T. Kato,
\textit{On nonlinear Schr\"odinger equations}, 
Ann.\ Inst.\ H. Poincar\'e Phys.\ Th\'eor.\ \textbf{46} (1987), 113--129.

\bibitem{LeCoz08-1}
S. Le Coz, 
\textit{A note on Berestycki--Cazenave's classical instability result for nonlinear Schr\"odinger equations}, 
Adv.\ Nonlinear Stud.\ \textbf{8} (2008), 455--463.


\bibitem{LeCoz08-2}
S. Le Coz,
\textit{Standing waves in nonlinear Schr\"odinger equations}, Analytical and Numerical Aspects of Partial Differential Equations, de Gruyter, Berlin, (2009), 151--192.


\bibitem{Lions84}
P.-L. Lions,
\textit{The concentration-compactness principle in the calculus of variations. The locally compact case, II},
Ann.\ Inst.\ H. Poincar\'e Anal.\ Non Lin\'eaire \textbf{1} (1984),  223--283.

\bibitem{Maeda08}
M. Maeda,
\textit{Stability and instability of standing waves for 1-dimensional nonlinear Schr\"odinger equation with multiple-power nonlinearity},
Kodai Math.\ J. \textbf{31} (2008), 263--271.


\bibitem{Ohta95-d}
M. Ohta,
\textit{Stability and instability of standing waves for one-dimensional nonlinear Schr\"odinger equations with double power nonlinearity},
Kodai Math.\ J. \textbf{18} (1995), 68--74.

\bibitem{Ohta95-1}
M. Ohta,
\textit{Instability of standing waves for the generalized Davey--Stewartson system},
Ann.\ Inst.\ H. Poincar\'e Phys.\ Th\'eor.\ \textbf{62} (1995), 69--80.


\bibitem{Ohta18}
M. Ohta,
\textit{Strong instability of standing waves for nonlinear Schr\"odinger equations with harmonic potential},
Funkcial.\ Ekvac.\ \textbf{61} (2018), 135--143.

\bibitem{OY15}
M. Ohta and T. Yamaguchi,
\textit{Strong instability of standing waves for nonlinear Schr\"odinger equations with double power nonlinearity},
SUT J. Math.\ \textbf{51} (2015), 49--58.

\bibitem{OY16}
M. Ohta and T. Yamaguchi, 
\textit{Strong instability of standing waves for nonlinear Schr\"odinger equations with a delta potential}, 
Harmonic analysis and nonlinear partial differential equations,
79--92, RIMS K\^{o}ky\^uroku Bessatsu, \textbf{B56}, Res.\ Inst.\ Math.\ Sci.\ (RIMS), Kyoto, 2016.


\bibitem{SS}
C. Sulem and P-L. Sulem,
The nonlinear Schr\"odinger equation: Self-focusing and wave collapse, 
Applied Mathematical Sciences, \textbf{139}, Springer-Verlag, New York, 1999.

\bibitem{Strauss77}
W. A. Strauss,
\textit{Existence of solitary waves in higher dimensions},
Comm.\ Math.\ Phys.\ \textbf{55} (1977), 149--162.



\bibitem{Weinstein8283}
M. I. Weinstein,
\textit{Nonlinear Schr\"odinger equations and sharp interpolation estimates}, 
Comm.\ Math.\ Phys.\ \textbf{87} (1982/83), 567--576.

\bibitem{Zhang02}
J. Zhang,
\textit{Cross-constrained variational problem and nonlinear Schr\"odinger equation}, 
Foundations of computational mathematics (Hong Kong, 2000), 457--469, World Sci.\ Publ., River Edge, NJ, 2002.


\end{thebibliography}
\end{document}